\documentclass[12pt]{amsart}
\usepackage[headings]{fullpage}
\usepackage{amssymb,amsmath,epic,eepic,epsfig}
\usepackage{epstopdf}
\usepackage{graphicx}
\usepackage{texdraw}
\usepackage{url}
\usepackage[bookmarks=true,%
    colorlinks=true,%
    linkcolor=blue,%
    citecolor=blue,%
    filecolor=blue,%
    menucolor=blue,%
    urlcolor=blue,%
    breaklinks=true]{hyperref}

\newtheorem{theorem}{Theorem}[section]
\newtheorem{proposition}[theorem]{Proposition}
\newtheorem{lemma}[theorem]{Lemma}
\newtheorem{corollary}[theorem]{Corollary}

\newtheorem{definition}[theorem]{Definition}
\newtheorem{remark}[theorem]{Remark}
\newtheorem{example}[theorem]{Example}

\def\BZ{\mathbb Z}

\def\BQ{\mathbb Q}
\def\BR{\mathbb R}

\def\calP{\mathcal P}

\def\calQ{\mathcal Q}

\def\calM{\mathcal{M}}

\def\l{\lambda}
\def\Ga{\Gamma}

\def\la{\langle}
\def\ra{\rangle}

\def\Ga{\Gamma}

\def\longto{\longrightarrow}

\def\PSL{\mathrm{PSL}}

\def\Vol{\mathrm{Vol}}

\def\Res{\mathrm{Res}}

\def\Isom{\mathrm{Isom}}
\def\vol{\mathrm{vol}}
\def\Comm{\mathrm{Comm}}

\def\qed{ $\sqcup\!\!\!\!\sqcap$}

\def\Re{\mbox{\rm{Re}}}

\begin{document}


\title[Constructing 1-cusped isospectral non-isometric 
hyperbolic 3-manifolds]{
Constructing 1-cusped isospectral non-isometric 
hyperbolic 3-manifolds}
\author{Stavros Garoufalidis}
\address{School of Mathematics \\
         Georgia Institute of Technology \\
         Atlanta, GA 30332-0160, USA \newline 
         {\tt \url{http://www.math.gatech.edu/~stavros }}}
\email{stavros@math.gatech.edu}
\author{Alan W. Reid}
\address{Department of Mathematics \\
        University of Texas \\
        1 University Station C1200 \\
        Austin, TX 78712-0257, USA \newline
         {\tt \url{http://www.ma.utexas.edu/users/areid }}}
\email{areid@math.utexas.edu}
\thanks{The authors were supported in part by National Science 
Foundation. \\
\newline
1991 {\em Mathematics Classification.} Primary 57N10. Secondary 57M25.
}

\date{August 1, 2016}


\begin{abstract}
We construct infinitely many examples of pairs of isospectral but
non-isometric $1$-cusped hyperbolic $3$-manifolds. These examples 
have infinite discrete spectrum and the same Eisenstein
series.  Our constructions are based on an application of Sunada's
method in the cusped setting, and so in addition our pairs are
finite covers of the same degree of a 1-cusped hyperbolic 3-orbifold
(indeed manifold) and also have the same complex length-spectra.
Finally we prove that any finite volume hyperbolic 3-manifold
isospectral to the figure-eight knot complement is homeomorphic to the
figure-eight knot complement. 
\end{abstract}

\maketitle

{\footnotesize
\tableofcontents
}


\section{Introduction}

Since Kac \cite{Kac} formulated the question: {\em Can you hear the shape
of a drum?}, there has been a rich history in 
constructing isospectral but non-isometric manifolds in various settings.
We will not describe this in any detail here, but 
simply refer the reader to \cite{Gor} for a survey.  The
main purpose of this note is to prove the following result (see also
Theorem \ref{main_restate} for a more detailed statement).

\begin{theorem}
\label{main}
There are infinitely many pairs of finite volume orientable 1-cusped 
hyperbolic 3-manifolds that are isospectral but non-isometric.
\end{theorem}

Since we are working with cusped hyperbolic 3-manifolds, the 
statement of the theorem requires some clarification. 
Indeed, one can reasonably ask, {\em what does isospectral mean for
cusped hyperbolic 3-manifolds}. We address this in Section 2, where we
indicate the differences with the closed case. Our examples appear to be
the first examples of 1-cusped hyperbolic 3-manifolds that are known to
be isospectral and non-isometric. On the other hand, 
there has been considerable interest in this for 
surfaces (both non-compact finite area, and infinite area convex
cocompact, see \cite{BroD} and the survey \cite{GPS}). In fact
in \cite{GPS}, they raise the problem  (Problem 1.2 of \cite{GPS})
of finding such examples in much more generality. Note that these papers use
the terms {\em isoscattering} or {\em isopolar}, but we prefer
to stick with isospectral.

Theorem \ref{main} is well-known for closed hyperbolic 3-manifolds, either 
using the arithmetic methods of \cite{Vig}, or the method of Sunada (\cite{Su}
and which we  recall below), as in \cite[p.225]{Reid}. As in this latter 
setting, our construction also uses Sunada's method, but we need
some additional control.  In addition to proving the
existence of infinitely many pairs of examples of isospectral
$1$-cusped hyperbolic 3-manifolds, we also give more concrete
examples of isospectral manifolds arising as low degree
covers of small volume 1-cusped hyperbolic 3-manifolds arising in 
the census of hyperbolic manifolds of 
\texttt{Snap} and \texttt{SnapPy} \cite{snap,snappy}.

One motivation for Theorem \ref{main} was  to investigate the nature of the 
discrete spectrum of $1$-cusped hyperbolic 3-manifolds.  There has been
considerable interest in this for non-compact surfaces of finite area
(see \cite{Mu}, \cite{Sa1} and \cite{Sa2} to name a few), but little 
seems known in dimension $3$.  We discuss this further in Section 8, and in
particular we prove the following.

\begin{theorem}
\label{main_fig8}
Let $M$ denote the complement of the figure-eight knot in $S^3$. Suppose
that $N$ is a finite volume hyperbolic 3-manifold which is isospectral with
$M$. Then $N$ is homeomorphic to $M$.
\end{theorem}

Indeed we also show that the first ten $1$-cusped orientable finite volume
hyperbolic 3-manifolds are determined by their spectral data (see Definition
2.1).


\section{What does isospectral mean for cusped manifolds?}

As remarked upon in the Introduction, since we are in the setting of
cusped orientable hyperbolic 3-manifolds, some clarification about the
statement of Theorem \ref{main} is required, and in this section we
explain what we mean by {\em isospectral $1$-cusped hyperbolic 3-manifolds}.  
Throughout this paper we will
restrict ourselves to only discussing the spectrum for $1$-cusped
hyperbolic 3-manifolds. This simplifies things, but much of what is
described in this section holds more generally, 
and similar statements can be made in the
presence of multiple cusps. We refer the reader to Chapters 4 and 6 of
\cite{EGM} or \cite{CoS} for a detailed discussion of the spectral theory
of cusped hyperbolic 3-manifolds.

\subsection{The spectrum of the Laplacian in the cusped setting} 
\label{sub.spec.lap}

Let $M={\bf H}^3/\Gamma$ be a 1-cusped,
orientable finite volume hyperbolic 3-manifold.
The spectrum of the Laplacian on the space $L^2(M)$
consists of {\em a discrete spectrum} (i.e. a collection of eigenvalues 
$0\leq \lambda_1\leq \lambda_2 \ldots$ where
each $\lambda_j$ has finite multiplicity), together with {\em a continuous
spectrum} (a copy of the interval $[1,\infty)$). 
Moreover, the discrete spectrum
consists of finitely many eigenvalues in $[0,1)$, together
with those eigenvalues embedded in the continuous spectrum.
However, unlike the closed setting,
in general, it is unknown as to whether the discrete spectrum is infinite
(we address this point in Section 2.3). 

The eigenfunctions
associated to eigenvalues in the discrete spectrum form an orthonormal system
and the closed subspace of $L^2(M)$ that they generate
is denoted by $L^2_{disc}(M)$. The orthogonal complement of
$L^2_{disc}(M)$ in $L^2(M)$ is denoted by $L^2_{cont}(M)$ and ``corresponds'' 
(in a way
that we need not make precise here) to the continuous spectrum (see
\cite{EGM} Chapters 4 and 6).

In the closed case, the Weyl law provides a way to prove that the
discrete spectrum is infinite (see \cite{EGM} Chapter
5). The precise analogue of this in the cusped setting is not
available, and this necessitates understanding a contribution from an 
Eisenstein series associated to the cusp of $M$.  To describe this further, 
conjugate $\Gamma$ so that a maximal peripheral subgroup
$P<\Gamma$ fixes infinity.  Fixing co-ordinates on 
${\bf H}^3 = \{w=(z,y): z\in {\bf C}, y\in {\bf R}_+\}$, we define 
the {\em Eisenstein series associated to the cusp at infinity} by:

$$E(w,s) = \sum\limits_{\gamma\in P\backslash\Gamma} y(\gamma w)^s,$$

\noindent where $y(p)$ denotes the $y$ co-ordinate of the point
$p \in{\bf H}^3$.
Now since $E(w,s)$ is $P$-invariant, an analysis of the  Fourier expansion 
at $\infty$ reveals a constant term of the form:

$$
y(w)^s + \phi(s)y(w)^{2-s},
$$

\noindent where $\phi(s)$ is the so-called {\em scattering
  function}. This is defined for $\Re(s)>2$ and has a meromorphic
extension to the complex plane.  The poles of $\phi(s)$ are also the
poles of the Eisenstein series and all lie in the half-plane
$\Re(s)<1$, except for at most finitely many in the interval $(1,2]$.
Moreover, if $t\in (1,2)$ is a pole, the residue $\psi = \Res_{s=t}E(w,s)$ 
is an eigenfunction with eigenvalue $t(2-t)$ (see \cite{CoS}). In
addition, if there is a pole at $s=2$, the residue will be an
eigenfunction with eigenvalue $0$ (\cite{CoS}). This subset of the
discrete spectrum arising from residues of poles of the Eisenstein
series is called the {\em residual spectrum}. If $t$ is a pole 
of $E(w,s)$ (equivalently $\phi(s)$) we define the {\em multiplictity}
at $t$ to be the order of the pole at $t$, plus the dimension of the eigenspace
in the case when $t$ contributes to the residual spectrum as described above.

The following definition is, in part, motivated by what spectral
information is required to determine the geometry in the cusped setting; e.g. 
the role of the scattering function and its poles is natural 
in the analogue of the Weyl law for cusped manifolds(see \cite{EGM} Theorem
6.5.4).  

\begin{definition}
\label{isospectral}
Let $M_1$ and $M_2$ be 1-cusped orientable hyperbolic 3-manifolds of finite
volume with associated scattering functions $\phi_1(s)$ and $\phi_2(s)$. 
Assume that the discrete spectrum of $M_1$ is infinite.
Say that $M_1$ and $M_2$ are isospectral if:

\begin{itemize}
\item $M_1$ and $M_2$ have the same discrete spectrum, counting 
multiplicities;
\item $\phi_1(s)$ and $\phi_2(s)$ have the same set of poles and 
multiplicities.
\end{itemize}
\end{definition}

\begin{remark} (1) For a 1-cusped orientable finite volume
hyperbolic 3-manifold $M$, its discrete spectrum, counting 
multiplicities, together with the set of poles and 
multiplicities of the scattering function will be referred to as its
spectral data.\\[\baselineskip]
(2)  For a multi-cusped orientable finite volume
hyperbolic 3-manifold $M$, the scattering function is a matrix (the 
scattering matrix), and in this 
case one takes the determinant of the scattering
matrix to obtain a function $\tau_M(s)$ that plays the role of $\phi(s)$
above.\\[\baselineskip]
(3) Continuing with the discussion of the role of the scattering
function in determining the geometry from spectral data,
it is shown in \cite{Mu} that an analogue
of Huber and McKean's results for compact surface holds. 
Namely, the spectral data in Definition 2.1 (in
the context of a non-compact
hyperbolic surface of finite area),
determines the length spectrum of the surface and vice versa (see Section
8 for a discussion of this for 1-cusped hyperbolic 3-manifolds). Moroever,
there are only finitely many hyperbolic surfaces with 
the given spectral data.\\[\baselineskip]
(4) In general the scattering determinant is hard to compute explicitly.
However, for arithmetic manifolds (and orbifolds) the scattering
determinant is related to Dedekind zeta functions of number fields.
For example, for $\PSL(2,{\Bbb Z})$ the poles of the scattering function
are related to the zeroes of $\zeta(s)$ (see \cite{Sa1}),
whilst for the Bianchi orbifolds with one cusp,
the scattering function is expressed in terms of the zeta function 
$\zeta_K(s)$ attached to the quadratic imaginary number field $K$
(see \cite{ES} or \cite[Chpt.8.3]{EGM}).
\end{remark}

\subsection{Manifolds with the same Eisenstein series}

The following lemma will be useful in our construction.  We fix
some notation. Let $M={\bf H}^3/\Gamma$ be a $1$-cusped orientable 
finite volume hyperbolic
3-manifold with finite covers $M_i={\bf H}^3/\Gamma_i$ ($i=1,2$) both
with one cusp and of same covering degree, $n$ say.  
Conjugate $\Gamma$ so that a maximal peripheral subgroup 
$P<\Gamma$ fixes $\infty$, and let $P_i=\Gamma_i\cap P$.
Denote the Eisenstein series associated to $M$, $M_1$ and
$M_2$ constructed in Section \ref{sub.spec.lap} by $E(w,s)$, $E_1(w,s)$ 
and $E_2(w,s)$ respectively.

\begin{lemma}
\label{sameeis}
Let $M$, $M_1$ and $M_2$ be as above. 
Then $E_1(w,s) = E_2(w,s)$. In particular $M_1$ and $M_2$ have the same 
scattering function.
\end{lemma}

\begin{proof} 
We begin the proof with a preliminary remark.
Suppose that $N={\bf H}^3/G$ is $1$-cusped and
is a finite covering of $M$.
We claim that a set of distinct coset 
representatives for $G$ in $\Gamma$ can be chosen from elements of $P$.
Briefly, since the preimage of the cusp of $M$
is connected (i.e. is the single cusp of $N$), we have must 
equality of indices $[\Gamma:G] = [P:P\cap G]$. Thus a 
collection of coset representatives for $P\cap G$ in $P$ also works 
as coset representatives $G$ in $\Gamma$.

Given this, let
$S = \{\delta_1, \ldots ,\delta_n\}\subset P$ be a set of distinct
(left) coset representatives for $\Gamma_1$ in $\Gamma$ and
$S' = \{\delta_1', \ldots ,\delta_n'\}\subset P$ be a set of distinct
coset representatives for $\Gamma_2$ in $\Gamma$.

Now any term in $E(w,s)$ has the form $y(\gamma w)^s$ for 
$\gamma \in \Gamma$ not fixing $\infty$.  Using the above decomposition
of $\Gamma$ as a union of cosets of both $\Gamma_1$ and $\Gamma_2$,
there exists 
$\gamma_1\in\Gamma_1$,  $\gamma_2\in \Gamma_2$ and $\delta_j\in S$,
$\delta_k'\in S'$ so that:

$$
\delta_j\gamma_1 = \gamma = \delta_k'\gamma_2.
$$

Since $\delta_j, \delta_k' \in P$ and $\gamma\notin
P$, it follows that $\gamma_1\notin P_1$ and $\gamma_2\notin P_2$
(otherwise $\gamma\in P$, contrary to the definition of the Eisenstein
series).  Using the coset decomposition of $\Gamma$, it follows
that $E(w,s)$ can be decomposed as a sum of terms of the form:

\begin{equation}
\tag{$\ast$}
\label{ast}
\sum\limits_{g\in P_1\backslash\Gamma_1} y(\delta_jgw)^s \qquad \text{and}
\qquad \sum\limits_{h\in P_2\backslash\Gamma_2} y(\delta_k'hw)^s.
\end{equation}

Since $\delta_j,\delta_k'\in P$, they act by translation on
${\bf H}^3$, and in particular the $y$-cordinate is unchanged by this; i.e.
$y(\delta_jgw) = y(gw)$ and 
$y(\delta_k'hw) = y(hw)$. Hence the terms in \eqref{ast}
above reduce to
$$\sum\limits_{g\in P_1\backslash\Gamma_1} y(gw)^s~\hbox{and}~
\sum\limits_{h\in P_2\backslash\Gamma_2} y(hw)^s.$$

So, putting all of this together, we have the following:

$$
nE_1(w,s) = E(w,s) = nE_2(w,s),
$$
which proves the lemma.
\end{proof}

\subsection{Ensuring the discrete spectrum is infinite} 

In this section we address the issue of ensuring that
the discrete spectrum is infinite. In particular we state
a result that 
can be proved using the methods of \cite{Ven1} (see
also the comments in \cite{EGM} at the end of Chapter 6.5). To state
the result we need to recall the following.

A fundamental dichotomy of Margulis for a finite volume hyperbolic 
manifold $M={\bf H}^3/\Gamma$, 
is whether $M$ is arithmetic or not. In the commensurability class of a
non-arithmetic manifold, there is a unique minimal element in the
commensurability class. This minimal element arises as 
${\bf H}^3/\Comm(\Gamma)$, where

$$
\Comm(\Gamma) = \{g\in\Isom({\bf H}^3):g\Gamma g^{-1}~\hbox{is commensurable
with}~ \Gamma\}
$$ 
is {\em the commensurator} of $\Gamma$.

The following can be proved following the methods in \cite{Ven1}. 
Note that in the statement of \cite{Ven1} Theorem 2, 
a certain matrix determinant is assumed to be non-vanishing. 
In our setting, since the 
manifold has one cusp, this matrix coincides with a function 
and can be shown to not be identically zero.

\begin{theorem}
\label{venkov}
Let $M={\bf H}^3/\Gamma$ be an orientable finite volume $1$-cusped 
non-arithmetic hyperbolic 3-manifold that is not the minimal element in its
commensurability class (i.e. $\Gamma\neq \Comm(\Gamma)$). Then the
discrete spectrum of $M$ is infinite.
\end{theorem}

We will make some further comments on the nature of the discrete spectrum 
(when it is known to be infinite) in Section 8.

\subsection{The complex length spectrum}
\label{sub.complex}

Let $M={\bf H}^3/\Gamma$ be an orientable finite volume hyperbolic
3-manifold. Given a loxodromic element $\gamma\in\Gamma$, the 
{\em complex translation length} of $\gamma$ is the complex number
$L_\gamma = \ell_\gamma+i\theta_\gamma$,
where $\ell_\gamma$ is the translation length of $\gamma$ and 
$\theta_\gamma\in[0,\pi)$ is the angle incurred in translating along the axis
of $\gamma$ by distance $\ell_\gamma$. 
The {\em complex length spectrum} of $M$ is defined to be the
collection of all complex translation lengths counted with multiplicities.

Given the discussion of the previous subsections, we now give a more
detailed statement of Theorem \ref{main}.

\begin{theorem}
\label{main_restate}
There are infinitely many pairs of finite volume orientable 1-cusped 
hyperbolic 3-manifolds that are isospectral but non-isometric.
In addition, our pairs have the following properties:
\begin{itemize}
\item cover a 1-cusped hyperbolic 3-manifold of the same degree,
\item have infinite discrete spectrum,
\item have the same Eisenstein series,
\item have the same complex length spectra.
\end{itemize}
\end{theorem}


\section{The Sunada construction in the $1$-cusped setting}
\label{sec.sunada}

Let $G$ be a finite
group and $H_1$ and $H_2$ subgroups of $G$. 
We say that $H_1$ and $H_2$ are {\em almost conjugate} 
if they are not conjugate in $G$ but
for every conjugacy class $C \subset G$ we have:

$$
|C\cap H_1| = |C\cap H_2|.
$$

If the above condition is satisfied, we call $(G,H_1,H_2)$ a 
{\em Sunada triple}, and $(H_1,H_2)$ an almost conjugate pair in $G$. 
We prove the following using 
Sunada's method \cite{Su} (cf. \cite{Ber, Br, Br1, PR}).

\begin{theorem}
\label{sunada}
Let $M={\bf H}^3/\Gamma$ be a 1-cusped finite volume orientable 
hyperbolic 3-manifold
that is non-arithmetic and 
the minimal element in its commensurability class.  Let $G$ be
a finite group, $(H_1,H_2)$ an almost conjugate pair in $G$, and
assume that $\Gamma$ admits a homomorphism onto $G$.  Assume that 
the finite covers $M_1$ and $M_2$ associated to the pullback subgroups of
$H_1$ and $H_2$ have $1$ cusp. 
Then $M_1$ and $M_2$ are isospectral, have the same
complex length spectra and are non-isometric.
\end{theorem}

\begin{proof} First, note that the manifolds $M_1$ and $M_2$
cannot be isometric, since
if there exists $g\in \Isom({\bf H}^3)$ with
$g\Gamma_1 g^{-1}=\Gamma_2$, then this implies that $g\in\Comm(\Gamma)$.
However, by assumption, $\Comm(\Gamma)=\Gamma$, and so
projecting to the finite group
$G$, we effect a conjugacy of the almost conjugate
pair $(H_1,H_2)$, a contradiction.

To prove isospectrality, there are two things that
need to be established; that both $M_1$ and $M_2$ have the same
infinite discrete spectrum with multiplicities, 
and that their scattering functions have the same poles with multiplicities.
Since $M_1$ and $M_2$ are $1$-cusped, and 
$[\Gamma:\Gamma_1]=[\Gamma:\Gamma_2]$, the latter follows
immediately from the fact that their Eisenstein series are the same
by Lemma \ref{sameeis}.

Regarding the former statement, Theorem \ref{venkov}
shows that the discrete spectrum is infinite for both $M_1$ and $M_2$, and
we deal with remaining statement about the discrete spectra
in a standard way following \cite{Su}. For completeness we sketch a
proof of this.

Now it can be shown that to
prove that $M_1$ and $M_2$ have the same discrete spectra with multiplicities,
it suffices to show that $L^2_{disc}(M_1)\cong L^2_{disc}(M_2)$.  
To see this we find it convenient to follow
\cite{PR} and we refer the reader to that paper for details.  We need
a lemma from \cite{PR} and this requires some notation.
Let $G$ be
a finite group, and $V$ is a $G$-module. Denote by $V^G$ the 
submodule of $V$ invariant under the $G$-action. The following is
Lemma 1 of \cite{PR}:

\begin{lemma}
\label{PRlemma}
Suppose $G$ is a finite group, $(H_1,H_2)$ an almost conjugate pair in $G$
and suppose that $G$ acts on the complex vector space $V$. Then there is an 
isomorphism $\iota: V^{H_1}\rightarrow V^{H_2}$, commuting with the action of 
any endomorphism $\Delta$ of $V$ for which the following diagram commutes.
$$
\begin{matrix}
V^{H_1}& \mathop{\longrightarrow}\limits^{\iota}&V^{H_2}\\
\Delta \Big\downarrow\quad &&\Big\downarrow  \Delta\\
\noalign{\vskip6pt}
V^{H_1}&\mathop{\longrightarrow}\limits^{\iota} &V^{H_2}
\end{matrix}
$$
\end{lemma}

Now let $M_0$ be the cover of $M$ corresponding to the kernel of the 
homomorphism to $G$. 
Taking $V$ to be $L^2_{disc}(M_0)$ in Lemma \ref{PRlemma}, $\Delta$ to be the 
Laplacian, and noting that
for $i=1,2$, $L^2_{disc}(M_i) = L^2_{disc}(M_0)^{H_i}$, 
it follows that $L^2_{disc}(M_1) \cong 
L^2_{disc}(M_2)$.

The proof that the manifolds have the same
complex length spectra follows that
given in \cite{Su}. 
\end{proof}

\begin{remark} 
As noted above, the method of Sunada \cite{Su} 
also produces pairs of finite volume hyperbolic 3-manifolds 
with the same complex length spectrum.  More generally, in
the case of closed hyperbolic 3-manifolds, the complex length spectrum 
is known to determine the spectrum of the Laplacian, see 
\cite[Thm.1.1]{Salvai}. This also holds for cusped hyperbolic manifolds, 
as can be seen from \cite[Thm.2]{Kelmer} for example.
\end{remark}

\begin{example}
\label{ex.triple}
For $p$ a prime, we denote by ${\bf F}_p$
the finite field of $p$ elements, and denote by $\PSL(2,p)$ the finite
group $\PSL(2,{\bf F}_p)$ (which of course are simple for $p>3$). It is
known that (see \cite{Gu} for example) for $p=7,11$
the groups $\PSL(2,p)$ contain almost conjugate pairs of subgroups of
index $7$ and $11$ respectively.
\end{example}

\begin{remark}
\label{rem.lessthan7}
In \cite{Gu}, it is shown that there are no examples of almost conjugate
(but not conjugate) subgroups of a finite group of index less than $7$.
Hence, $7$-fold covers are the smallest index covers for which the
Sunada construction can be performed. 
\end{remark}

Given the previous set up, we can 
now prove the following straightforward proposition that is the key
element in our construction. 
We require a preliminary definition.
Following Riley \cite{Ri} if $M={\bf H}^3/\Gamma$ is an orientable
finite volume 1-cusped hyperbolic 3-manifold, $P<\Gamma$ a fixed
maximal peripheral subgroup and $\rho :\pi_1(M)\rightarrow \PSL(2,p)$
a representation, then $\rho$ is called a {\em $p$-rep} if $\rho(P)$
is non-trivial and all non-trivial elements in $\rho(P)$
are parabolic elements of $\PSL(2,p)$. In which case, $\rho(P)$ is easily
seen to have order $p$. More generally if $M$ has more than 1-cusp we
call $\rho$ a $p$-rep of $\pi_1(M)$ if the image of all maximal
peripheral subgroups satisfies the same condition as above. 

\begin{proposition}
\label{key}
Let $M={\bf H}^3/\Gamma$ be an orientable non-arithmetic finite volume
1-cusped hyperbolic 3-manifold that is the 
minimal element in its commensurability class. Suppose that $\rho$ is
a $p$-rep of $\Gamma$ onto $G=\PSL(2,7)$ or $\PSL(2,11)$. Then $M$
has a pair of 1-cusped isospectral but non-isometric covers of degree
$7$ or $11$ respectively. In addition this pair of manifolds
have the same complex length spectra.
\end{proposition}

\begin{proof} 
Let $M_i={\bf H}^3/\Gamma_i$ ($i=1,2$), be the covers of
$M$ corresponding to the almost conjugate pair in Example \ref{ex.triple}
above in either of the cases $p=7,11$. 

Once we establish that $M_1$ and $M_2$
both have 1 cusp, that
$M_1$ and $M_2$ are isospectral and non-isometric follows
from Theorem \ref{sunada}.  This also shows that they have the same complex
length spectra.
We deal with the case of $p=7$, the case of $p=11$
is exactly the same.

Let $P$ denote a fixed maximal peripheral subgroup of $\Gamma$. 
For $i=1,2$, let $P_i=\Gamma_i\cap P$.
We claim that for $i=1,2$, $[P:P_i]=7$. This implies that the
covers $M_1$ and $M_2$ have one cusp, for then the degree
of the cover on a cusp torus of $M_i$ to the cusp of $M$ is $7$ to $1$, ie
$M_i$ can have only one cusp.

To prove the claim, since the epimorphism $\rho$ is a $p$-rep, the image of 
$P$ consists of parabolic elements of $\PSL(2,p)$, and as remarked upon above,
such subgroups have order $7$. On the
other hand, $H_1$ and $H_2$ have index $7$ in $\PSL(2,7)$, and since
$\PSL(2,7)$ has order $168$, the subgroups $H_1$ and $H_2$ both
have order $24$, which is co-prime to $7$.  
It follows from this that $\rho(P_i)=1$, so that $[P:P_i]=7$,
and this completes the proof.\end{proof}

We close this section by making the following observation. This will 
be helpful in computational aspects carried out in Section 7. 

Suppose that $M$ is a 1-cusped hyperbolic $3$-manifold and 
$\rho: \pi_1(M) \rightarrow \PSL(2,p)$ a representation. We will say that
$\rho$ is a $p$-good-rep if $\rho$ is an epimorphism and there exists
a pair of non-conjugate $p$-index subgroups $H_i$ of $\PSL(2,p)$ with
the following property: if $M_i$ is the cover of $M$ obtained from $H_i$, 
then $M_i$ is 1-cusped for $i=1,2$ and $M_1$ is not isometric to $M_2$. 
We are interested in $p=7,11$.

\begin{lemma}
\label{lem.explicitH}
Fix $p=7,11$. If $H_1$ and $H_2$ are non-conjugate index $p$ subgroups of 
$\PSL(2,p)$, then $(H_1,H_2)$ is a Sunada pair in $\PSL(2,p)$. 
\end{lemma}

\begin{proof}
This can be done efficiently in \texttt{magma}, since
a computation reveals that for $p=7,11$, the group $\PSL(2,p)$
has only two subgroups of index $p$, up to conjugation. Since $\PSL(2,p)$
has a Sunada pair, if follows that the above pair of subgroups is the
unique Sunada pair, up to conjugation. Moreover, $H_1$ and $H_2$ are 
interchanged by the outer automorphism group $\mathrm{Out}(\PSL(2,p))=\BZ/2\BZ$.
\end{proof}

\begin{corollary}
\label{cor.good}
Every $p$-good rep for $p=7,11$ is a $p$-rep.
\end{corollary}


\section{An example: covers of a knot complement in $S^3$}
\label{sec.ex1}


In the next section we will prove Theorem \ref{main}. It is instructive
in this section to present an example of Proposition
\ref{key}, as some of the methods used in this example
will be employed below. We discuss the method in a more general
framework in Section \ref{sec.2methods}.

Let $K$ be the knot $K11n116$ of the Hoste-Thistlethwaite table 
shown in Figure \ref{f.11n116}. $K$ is known as $11n114$ in the 
\texttt{Snap} census \cite{snap}, $11_{298}$ in
the \texttt{LinkExteriors} table, $t12748$ in the 
\texttt{OrientableCuspedCensus} and $K8_{297}$
in the \texttt{CensusKnots}.

\begin{figure}[htpb]
\begin{center}
\includegraphics[height=0.20\textheight]{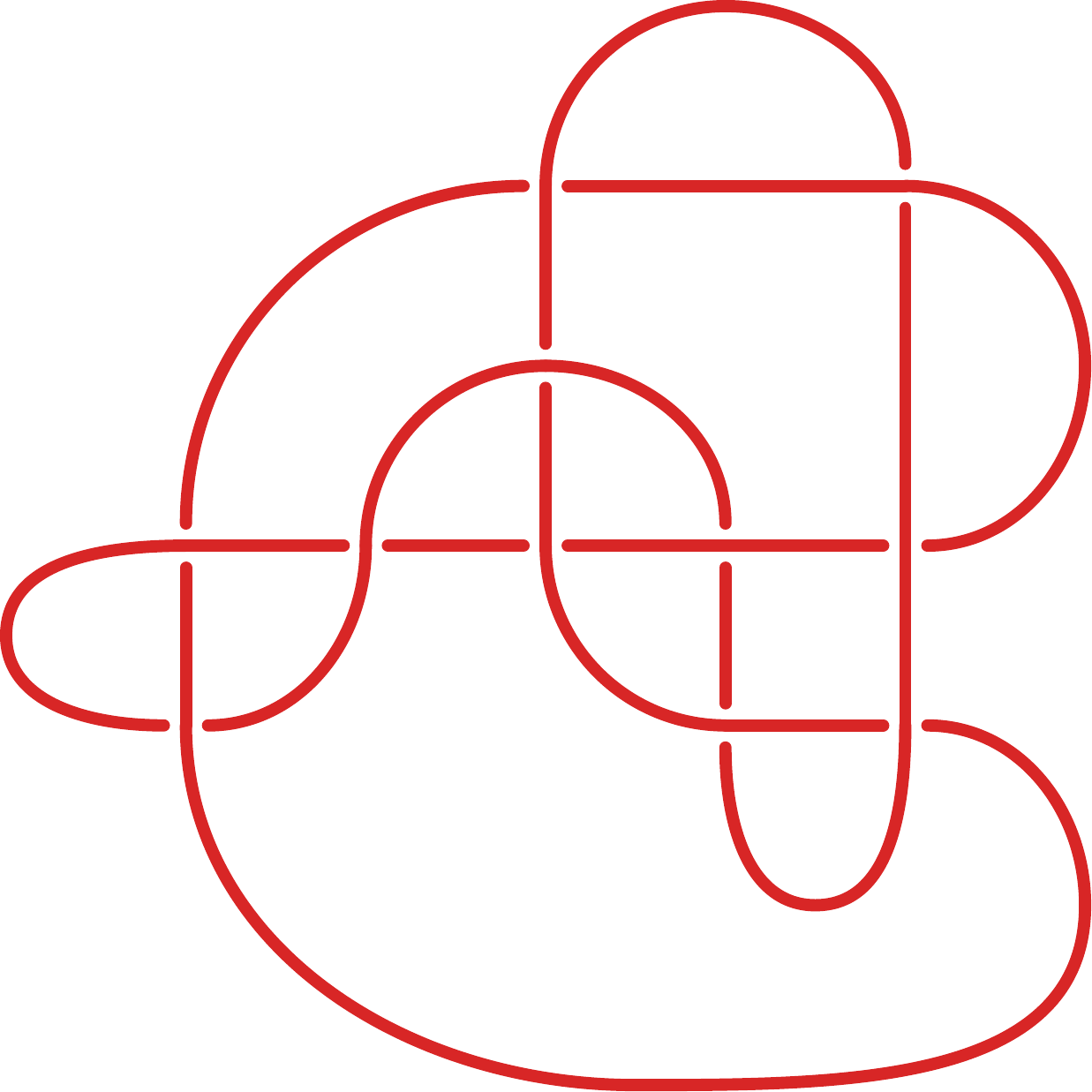}
\caption{The knot $K11n116$.}
\label{f.11n116}
\end{center}
\end{figure}

Using \texttt{Snap}, the manifold $M=S^3\setminus K={\bf H}^3/\Gamma$ has
a decomposition into $8$ ideal tetrahedra, has volume
$7.754453760202655\ldots$ and invariant trace field $k={\bf Q}(t)$
where $t=0.0010656 - 0.9101192 i$ is a root of the irreducible polynomial
$$
p(x)=x^8-3x^7+5x^6-3x^5+2x^4+2x^3+2x+1.
$$
Note that the discriminant of this polynomial is $156166337$, a prime, and
so this is the discriminant of $k$. Hence the ring of integers of $k$
(denoted $R_k$) coincides with $\BZ[t]$.

\texttt{Snap} shows that the geometric representation of $\Gamma$
has traces, lying in $R_k$ (see below).
In \cite{GHH} it is shown that $\Gamma=\Comm(\Gamma)$, and so we are in a 
position to apply Proposition \ref{key}.

\subsubsection{7-fold covers}

From above $7$ is unramified in $k/{\bf Q}$ (since $7$ does
not divide the discriminant of $k$), and using Pari \cite{pari}
for example, it can be shown that the ideal $(7)=7R_k$ factors as a product
${\calP}_1{\calP}_2{\calP}_3$ of prime ideals ${\calP}_i$ for $i=1,2,3$
of norm $7$, $7^2$ and $7^5$ respectively. 
$k$ has class number $1$, so all ideals are principal, and in the 
above notation, the prime ideal ${\calP}_1$ coincides with $(t-1)$. 

We will use the prime ideal ${\calP}_1$ (henceforth
denoted simply by $\calP$) to construct a $p$-rep as in Proposition
\ref{key}. To that end, we need to identify a particular conjugate of
$\Gamma$ with matrix entries in $R_k$. 
\texttt{Snap} yields the following presentation of $\Gamma$:

$$
\Gamma = \la a, b, c \,\,| \,\, aaCbAccBB, \, aacbCbAAB \ra
$$
with peripheral structure
$$
\mu= CbAcb, \qquad \lambda= AAbCCbacb \,,
$$
where, as usual $A=a^{-1}$, $B=b^{-1}$ and $C=c^{-1}$. Using \texttt{Snap} it
can shown that $\Gamma$ can be taken to be a subgroup of
$\PSL(2,R_k)$ represented by matrices as follows (note that from the 
irreducible polynomial of $t$ we see that $t$ is a unit):

$$ a =
\left(
\begin{array}{cc}
 -t^2+t-1 & t^7-3 t^6+4 t^5-t^4+t^2-t \\
 -t^2+t-1 & 0 \\
\end{array}
\right)$$

$$b =
\left(
\begin{array}{cc}
 -t^7+2 t^6-2 t^5-3 t^3+2 t^2-3 t-1 & t^6-2 t^5+t^4+3 t^3-2 t^2+3 t+2 \\
 -t^7+3 t^6-5 t^5+4 t^4-4 t^3+2 t^2-2 t-1 & t^7-3 t^6+5 t^5-4 t^4+4 t^3-t^2+t+2 \\
\end{array}
\right)$$

$$c =
\left(
\begin{array}{cc}
 -t^6+4 t^5-8 t^4+7 t^3-5 t^2-t & -2 t^7+7 t^6-14 t^5+15 t^4-12 t^3+t^2+3 t-1 \\
 t^5-3 t^4+4 t^3-3 t^2+t & -t^7+4 t^6-9 t^5+11 t^4-9 t^3+3 t^2+t-2 \\
\end{array}
\right)$$

The meridian and longitude are given by

$$\mu =
\left(
\begin{array}{cc}
 t^7-4 t^6+8 t^5-8 t^4+5 t^3-2 t & -t^7+2 t^6-3 t^5+t^4-2 t^3-4 t^2-2 t-1 \\
 t^7-4 t^6+9 t^5-11 t^4+10 t^3-3 t^2+3 & -t^7+4 t^6-8 t^5+8 t^4-5 t^3+2 t-2 \\
\end{array}
\right)$$

$$\lambda =
\left(
\begin{array}{cc}
 -2 t^7+6 t^6-10 t^5+7 t^4-7 t^3+3 t^2-8 t-1 & 2 t^7-9 t^6+18 t^5-19 t^4+15 t^3-11 t^2+3 t+6 \\
 6 t^7-20 t^6+38 t^5-35 t^4+31 t^3-t^2-t+18 & 2 t^7-6 t^6+10 t^5-7 t^4+7 t^3-3 t^2+8 t-1 \\
\end{array}
\right)$$

Now let $\rho_7: \Gamma \rightarrow \PSL(2,7)$
denote the $p$-rep obtained by reducing entries of these matrices
modulo $\calP$. A computation gives:
$$
\rho_7(a) =
\left(
\begin{array}{cc}
 6 & 1 \\
 6 & 0 \\
\end{array}
\right)
\qquad
\rho_7(b) =
\left(
\begin{array}{cc}
 1 & 6 \\
 3 & 5 \\
\end{array}
\right)
\qquad
\rho_7(c) =
\left(
\begin{array}{cc}
 3 & 4 \\
 0 & 5 \\
\end{array}
\right)
$$
and
$$
\rho_7(\mu) =
\left(
\begin{array}{cc}
 0 & 4 \\
 5 & 5 \\
\end{array}
\right)
\qquad
\rho_7(\lambda) =
\left(
\begin{array}{cc}
 2 & 5 \\
 1 & 3 \\
\end{array}
\right)
$$
We now check that $\rho_7$ is onto. To see this, note that
$T=\rho_7(aB) = \left(
\begin{array}{cc}
 -1 & 0 \\
 2 & -1 \\
\end{array}
\right)$
and performing the conjugation 
$\rho_7(a)T\rho_7(A)$ gives the matrix
$\left(
\begin{array}{cc}
 -1 & 2 \\
 0 & -1 \\
\end{array}
\right)$ \,.

Finally, after taking powers of these elements we see that
$\rho_7(\Gamma)$ contains the elements
$ \left(
\begin{array}{cc}
 1 & 0 \\
 1 & 1 \\
\end{array}
\right)$
and
$\left(
\begin{array}{cc}
 1 & 1 \\
 0 & 1 \\
\end{array}
\right).
$
These clearly generate $\PSL(2,7)$, and we are now in
a position to apply Proposition \ref{key} to complete the construction
of examples.

\subsubsection{11-fold covers}

$11$ is also unramified in $k/\BQ$ and $(11)=\calQ_1 \calQ_2 \calQ_3$
where $\calQ_i$ for $i=1,2,3$ are prime ideals of norm $11$, $11$ and $11^6$. 
Moreover, we can take $\calQ_1=(t+1)$ and $\calQ_2=(t^2-t-1)$.

Let $\rho'_{11},\rho''_{11}: \Gamma \longto \PSL(2,11)$
denote the $p$-reps obtained by reducing entries of these matrices
modulo $\calQ_1$ and $\calQ_2$ respectively. 
A computation gives:

{\small
$$
\rho'_{11}(a) =
\left(
\begin{array}{cc}
 8 & 4 \\
 8 & 0 \\
\end{array}
\right)
\qquad
\rho'_{11}(b) =
\left(
\begin{array}{cc}
 1 & 9 \\
 9 & 5 \\
\end{array}
\right)
\qquad
\rho'_{11}(c) =
\left(
\begin{array}{cc}
 9 & 3 \\
 10 & 1 \\
\end{array}
\right)
$$
}
{\small
$$
\rho'_{11}(\mu) =
\left(
\begin{array}{cc}
 9 & 6 \\
 9 & 0 \\
\end{array}
\right)
\qquad
\rho'_{11}(\lambda) =
\left(
\begin{array}{cc}
 9 & 6 \\
 9 & 0 \\
\end{array}
\right)
$$
}
and
{\small
$$
\rho''_{11}(a) =
\left(
\begin{array}{cc}
 9 & 6 \\
 9 & 0 \\
\end{array}
\right)
\qquad
\rho''_{11}(b) =
\left(
\begin{array}{cc}
 4 & 36 \\
 12 & 12 \\
\end{array}
\right)
\qquad
\rho''_{11}(c) =
\left(
\begin{array}{cc}
 32 & 12 \\
 28 & 4 \\
\end{array}
\right)
$$
}
{\small
$$
\rho''_{11}(\mu) =
\left(
\begin{array}{cc}
 32 & 0 \\
 32 & 32 \\
\end{array}
\right)
\qquad
\rho''_{11}(\lambda) =
\left(
\begin{array}{cc}
 32 & 0 \\
 28 & 32 \\
\end{array}
\right)
$$
}
Note that $\rho'_{11}$ and $\rho''_{11}$ are not intertwined by an automorphism
of $\PSL(2,11)$ since $\rho'_{11}(\mu)=\rho''_{11}(\lambda)$ but
$\rho'_{11}(\mu) \neq \rho''_{11}(\lambda)$.

\begin{remark} 
The construction of closed examples in \cite{Reid} 
arise from Dehn surgery on the knot $9_{32}$ 
(a construction that we extend below). 
Proposition \ref{key} can be applied to show that
examples of isospectral 1-cusped manifolds arise as 
$11$-fold covers of $S^3\setminus 9_{32}$.  The examples constructed
above have much smaller volume and so are perhaps more interesting.
\end{remark}


\section{Proof of Theorem \ref{main}: Infinitely many examples}
\label{sec.mainproof}

In this section we complete the proof of Theorem \ref{main} 
by exhibiting infinitely
many examples.  This builds on the ideas of \cite[Sec.3]{Reid}  and
Section 4.

\subsection{A lemma}

Using ideas from \cite{Reid} together with Proposition \ref{key}, we
will prove the following. This will complete the proof of Theorem
\ref{main}, given the existence of a 2-cusped manifold as in Lemma
\ref{1cuspedminimal} (which we exhibit in Subsection 5.2).

\begin{lemma}
\label{1cuspedminimal}
Let $M={\bf H}^3/\Gamma$ be an orientable non-arithmetic finite volume
2-cusped hyperbolic 3-manifold that is the 
minimal element in its commensurability class. Suppose that $\rho$ is
a $p$-rep of $\Gamma$ onto $G=\PSL(2,7)$ or $\PSL(2,11)$. Then there
are infinitely many Dehn surgeries $r=p/q$ on one cusp of $M$ so that
the resultant manifolds $M(r)$ are hyperbolic and 
have 1-cusped covers that are isospectral
but non-isometric.
\end{lemma}

\begin{proof} 
We will deal with the case of $G=\PSL(2,7)$, the other
case is similar.  Associated to the two cusps of $M$ we
fix two peripheral subgroups $P_1$ and $P_2$, and we will perform
Dehn surgery on the cusp associated to $P_2$, thereby preserving parabolicity
of the non-trivial elements of $P_1$ after Dehn surgery. 

Fix a pair
of generators $\mu$ and $\lambda$ for $P_2$.
By $p/q$-Dehn surgery on the cusp associated to $P_2$ we mean that
the element $\mu^p\lambda^q$ is trivialized. We denote the result of
$p/q$-Dehn surgery by $M(p/q)$.  Note that for sufficiently large
$|p|+|q|$, the resultant surgered manifolds will be 1-cusped hyperbolic
manifolds and will still be the minimal elements in their commensurability 
class (see Theorem 3.2 of \cite{Reid}).

Since $\rho$ is a $p$-rep, $\rho(P_2)$ is non-trivial. 
Performing $p/q$-Dehn surgery on the cusp associated to $P_2$,
if we can arrange that $\rho(\mu^p\lambda^q)=1$, 
then the $p$-rep $\rho$ will factor through $\pi_1(M(p/q))$, 
thereby inducing a $p$-rep of $\pi_1(M(p/q))$.

Now $\rho(P_2)$ is a cyclic subgroup $C=\la x \ra$ of order $7$. Hence there
are integers $s,t\in\{0,\pm 1,\pm 2, \pm 3\}$ (not both zero)
so that $\rho(\mu)=x^s$ and $\rho(\lambda)=x^t$. Hence we
need to find infinitely many co-prime pairs $(p,q)$ which satisfy
$ps+qt=7d$ with $s,t$ as above and for integers $d$.  This is easily 
arranged by elementary number theory. For example, 
if exactly one of $\rho(\mu)$ or $\rho(\lambda)$ is trivial (say 
$\rho(\lambda)$), then we can choose integers $p=7n$
and $q$ coprime to $7n$ will suffice to prove the lemma in this case.
If both $s,t\neq 0$, a simlar argument holds. For example suppose that 
$s=t=2$. Then choosing $q=1$ and $p$ an integer of the form $7a-1$ will work.

Thus we have constructed infinitely many 1-cusped hyperbolic 3-manifolds
with a $p$-rep onto $\PSL(2,7)$ and so the proof is complete by an
application of Proposition \ref{key}.
\end{proof}

\subsection{A 2-component link--$9_{34}^2$}
\label{sub.ex3}


From \cite{GHH} the 2-component link $L=9^2_{34}$ of Rolfsen's table
(which is the link $9a62$ in the \texttt{Snap} census and $L9a21$ in the 
Hoste-Thistlethwaite table) shown in Figure 
\ref{f.9a62} has the property that $M=S^3\setminus L$ is the minimal 
element in its commensurability class. 
 
\begin{figure}[htpb]
\begin{center}
\includegraphics[height=0.20\textheight]{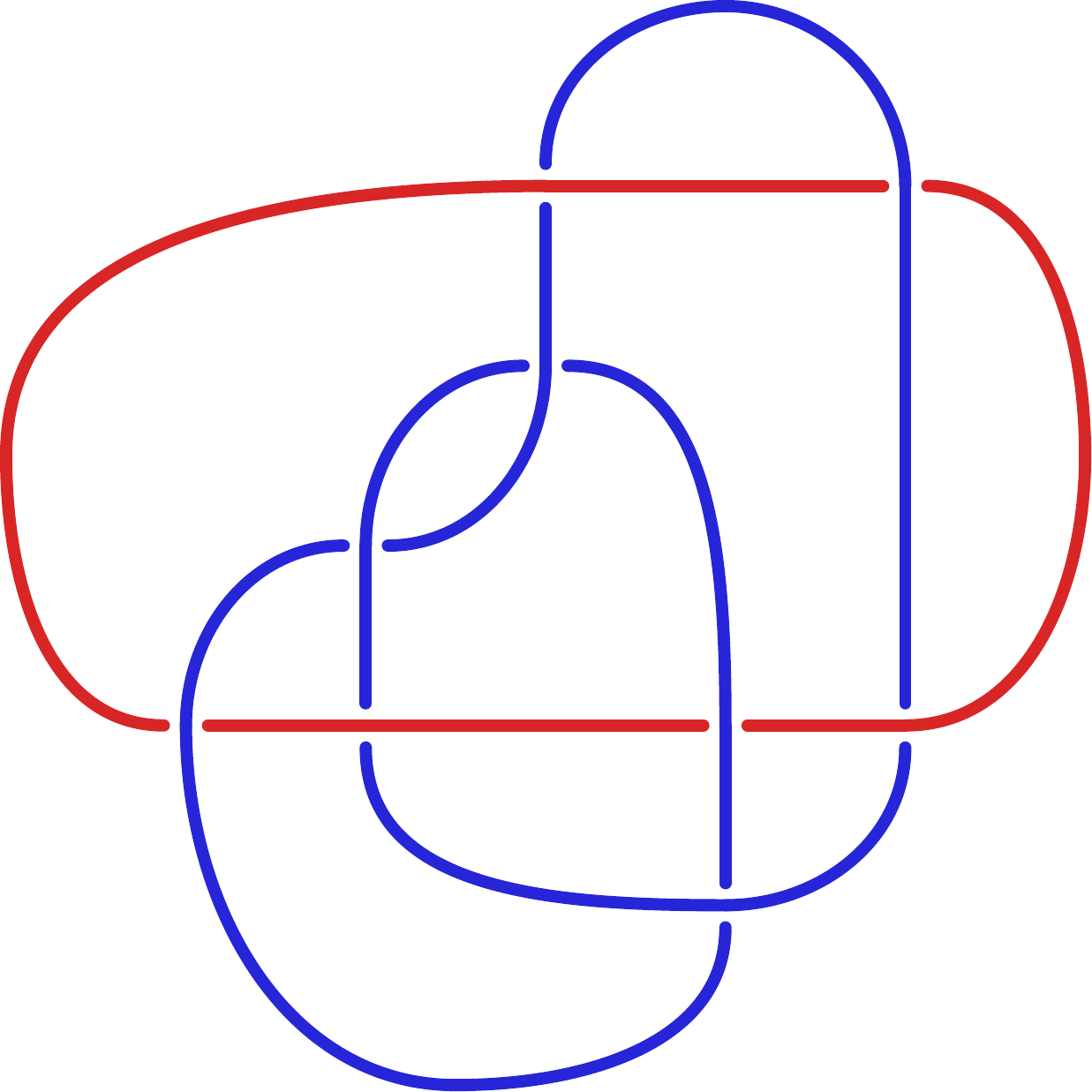}
\caption{The link $9^2_{34}$.}
\label{f.9a62}
\end{center}
\end{figure}

The link complement has volume
approximately $11.942872449472\ldots$ and invariant trace-field $k$ generated
by a root $t$ of:
$$
p(x)=x^{10} - x^9 - x^8 - x^7 + 6x^6 + x^5 - 3x^4 - 4x^3 + 2x^2 + 2x - 1.
$$
As can be checked using Pari, $(7)=\calP_1 \calP_2 \calP_3 \calP_4 \calP_5$
where $\calP_i$ for $i=1,\dots,5$ are prime ideals of norm $7$, $7$, $7^2$,
$7^3$ and $7^3$. Moreover, we can take $\calP_1=(t+1)$.
The fundamental group has presentation
$$
\Ga=\la a, b, c \,\,| \,\,
a B A C b c c a b C C B c a b A c b, \, a b A c b a C C B c c A B C
\ra
$$
with peripheral structure
$$
(\mu_1,\l_1)=(b, B B A c b a C C), \qquad 
(\mu_2,\l_2)=(B C, a B A C b c c a C C b c c B A C b)
\,.
$$
Following the ideas above 
it can be shown that the faithful discrete representation of
$\pi_1(M)$ can be conjugated to lie in $\PSL(2,R_k)$ and that reducing
modulo $\calP_1$ provides a $p$-rep onto $\PSL(2,7)$ given by
$$
\rho(a)=\left(
\begin{array}{cc}
 3 & 5 \\
 0 & 5 \\
\end{array}
\right), \qquad 
\rho(b)=
\left(
\begin{array}{cc}
 3 & 1 \\
 5 & 2 \\
\end{array}
\right),  \qquad 
\rho(c)=
\left(
\begin{array}{cc}
 3 & 6 \\
 6 & 3 \\
\end{array}
\right)
$$
$$
\rho(\mu_1)=\left(
\begin{array}{cc}
 3 & 1 \\
 5 & 2 \\
\end{array}
\right),  \qquad 
\rho(\l_1)=
\left(
\begin{array}{cc}
 4 & 6 \\
 2 & 5 \\
\end{array}
\right), \qquad
\rho(\mu_2)=
\left(
\begin{array}{cc}
 5 & 6 \\
 2 & 4 \\
\end{array}
\right), \qquad 
\rho(\l_2)=
\left(
\begin{array}{cc}
 3 & 6 \\
 2 & 2 \\
\end{array}
\right)
$$
Moreover, fixing a cusp, $\rho$ can be conjugated to a representation
such that the meridian and longitude pair of both map to 
$\left(
\begin{array}{cc}
 1 & -1 \\
 0 & 1 \\
\end{array}
\right)$
Choosing $p=-(7n+1)$ (sufficiently large) and $q=1$ provides explicit
Dehn surgeries as given by Lemma \ref{1cuspedminimal}.


\section{Two methods to construct Sunada pairs}
\label{sec.2methods}

We now discuss two methods for implementing Proposition \ref{key}.
In Section \ref{sec.ex1}, an example of a minimal knot complement
was used to build examples (we will refer to this example as Example 1).
The framework for this was to
reduce the geometric representation, defined over a localization
of the ring of integers of a number field, modulo a prime of norm $7$ or $11$.
We shall call this {\em Method G}. A second method (which we refer to
{\em Method R}), mentioned at the end
of Section \ref{sec.sunada}, is to compute all $p$-good reps for $p=7,11$. 
Each method has its own merits. Method R can be implemented efficiently by 
\texttt{magma} and \texttt{SnapPy} to search over lists of manifolds. 
Method G (which involves exact arithmetic computations) requires a combination 
of \texttt{Snap}, \texttt{SnaPy}, \texttt{pari} and \texttt{sage} and a lot of 
cutting and pasting, but produces infinitely many 1-cusped examples.

Let us describe Method G in more detail. We start with a 
cusped orientable hyperbolic $3$-manifold $M$. Its geometric representation 
$$
\pi_1(M) \longto \PSL(2,R)
$$
can be defined over a subring $R$ of an extension of the invariant trace-field.
In many cases, this is actually
contained in the invariant trace-field
$k$ (e.g. for knots in integral homology 3-spheres). If we can find a prime
ideal $\calP$ in $R$ of norm $7$ or $11$ which is not inverted in $R$,
then we can reduce the geometric representation of $M$ to get a 
representation $\rho: \pi_1(M) \longto \PSL(2,p)$ for $p=7$ or $p=11$.
We can further check that $\rho$ is a $p$-rep. If we can also compute
the commensurator of $\pi_1(M)$, then we can apply Proposition \ref{key}.

Before we get into the details, let us recall that (Hoste-Thistlethwaite
and Rolfsen) tables of hyperbolic knots are available from \texttt{SnapPy} 
\cite{snappy} and from \texttt{Snap} \cite{snap}. A consistent conversion 
between these tables is provided by \texttt{SnapPy} \cite{snappy}.



\section{More examples}
\label{sec.examples}

\subsection{Example 1 via Method R}
\label{sub.ex1}


Consider the knot $K=K11n116$ from Figure \ref{f.11n116} of Section 
\ref{sec.ex1}.


Setting $M=S^3\setminus K$, \texttt{magma} computes that $\pi_1(M)$ has $4$ 
epimorphisms in $\PSL(2,7)$ and
two of them are $7$-good reps. (corresponding to those we found in Section 4). 
The corresponding pair $M_1$ and $M_2$
of index $7$ covers are isospectral and non-isometric.
We can also confirm that $M_1$ and $M_2$ are not isometric using the
{\em isometry signature} (a complete invariant) of \cite{census-tet}.
As shown by \texttt{magma} 
both have common homology $\BZ/2 + \BZ/110 + \BZ$.  


\texttt{SnapPy} computes that $M$ has $42$ 11-fold covers. Of those, $8$ have
a total space with one cusp, and among those, we find $11$-good covers:
there is one
pair of covers with homology $\BZ/2 + \BZ/210 + \BZ$ and another pair
with homology $\BZ/2 + \BZ/406 + \BZ$ and non-isometric total spaces for either
pair. These pairs $(M'_1,M'_2)$ and $(M''_1,M''_2)$ 
are built from the epimorphisms $\rho'_{11}$ and $\rho''_{11}$.
$(M'_1,M'_2)$ and also $(M''_1,M''_2)$ 
are isospectral.

\subsection{Example 2: covers of the manifold v2986 via Method G}
\label{sub.ex2}


Let $M={\bf H}^3/\Gamma$ denote the manifold from the \texttt{Snap} census
$v2986$. \texttt{SnapPy} confirms that $M$ is not a knot complement in $S^3$ 
(since it can be triangulated using 7 ideal tetrahedra and is not isometric
to a manifold in \texttt{CensusKnots}, the complete list of hyperbolic
knots with at most 8 tetrahedra) but it does have $H_1(M;\BZ)=\BZ$. 
The volume of $M$ is approximately $6.165768948\ldots$ (which is less 
than the previous example).  Again
from \cite{GHH}, we have that $\Gamma=\Comm(\Gamma)$. \texttt{Snap} gives the
the following presentation of the fundamental group $\Ga$ 
$$
\Ga=\la a, b, c \,\,| \,\,
 a c b C B a B A c, \,  a b c b b A A C 
\ra
$$
with peripheral structure
$$
\mu= C , \qquad \l= B C a b A A  \,.
$$
From \texttt{Snap} we
see that $\Gamma$ has integral traces and has invariant trace-field generated
by a root of the polynomial
$$
p(x)=x^8-2x^7-x^6+4x^5-3x^3+x+1 
$$
Using Pari, we get a decomposition $(7)=\calP_1 \calP_2 \calP_3$
into prime ideals $\calP_1$, $\calP_2$ and $\calP_3$ of norm $7$, $7^3$ and 
$7^4$. Moreover, we can take $\calP_1=(t^3-t-1)$.
The geometric representation is still defined over $R_k$, and its
reduction $\rho_7: \Gamma \longto \PSL(2,7)$ modulo $\calP_1$ is given by:
$$
\rho_7(a) =
\left(
\begin{array}{cc}
 10 & 4 \\
 4 & 8 \\
\end{array}
\right)
\qquad
\rho_7(b) =
\left(
\begin{array}{cc}
 0 & 8 \\
 6 & 12 \\
\end{array}
\right)
\qquad
\rho_7(c) =
\left(
\begin{array}{cc}
 4 & 2 \\
 6 & 12 \\
\end{array}
\right)
$$
and
$$
\rho_7(\mu) =
\left(
\begin{array}{cc}
 12 & 12 \\
 8 & 4 \\
\end{array}
\right)
\qquad
\rho_7(\lambda) =
\left(
\begin{array}{cc}
 8 & 6 \\
 4 & 4 \\
\end{array}
\right)
$$
As before, one can check that $\rho_7$ is onto, and thereby construct 
isospectral covers with one cusp using Proposition \ref{key}.

\subsection{Example 3: covers of knot complements with at most 
8 tetrahedra via Method R}

Of the 502 hyperbolic knots in \texttt{CensusKnots} with at most 8
ideal tetrahedra, \texttt{SnapPy} computes that
the following 11 have trivial isometry group:

$$
K8_{226}, K8_{252}, K8_{270}, K8_{277}, K8_{287}, K8_{290},
K8_{292}, K8_{293}, K8_{296}, K8_{297}, K8_{301}
$$
Note that $K8_{297}$ is the knot of Example 1. 
\texttt{Snap} confirms that all of these knots have no hidden symmetries.
Of the above 11 knots, \texttt{magma} finds that the following 8 have
at least one $7$-good-rep:
$$
K8_{252}, K8_{270}, K8_{277}, K8_{290}, K8_{292}, K8_{293}, K8_{297}, K8_{301}
$$
and all 11 have at least one $11$-good-rep. 

\subsection{Example 4: the list of 1-cusped manifolds of \cite{GHH}
via Method R}



\cite{GHH} gave a list of $13486$ hyperbolic manifolds with at least one
cusp, along with their hidden symmetries. Of those with no hidden symmetries, 
$1252$ have one cusp, $1544$ have two cusps and $106$ have four cusps.

There are 6 manifolds with one cusp and no hidden symmetries and with
at most 7 ideal tetrahedra in the above list:
$$
v2986, v3205, v3238, v3372, v3398, v3522  
$$
\texttt{magma} computes that all $7$ of those manifolds have $7$-good reps,
and that the following $3$
$$
v3205, v3238, v3522
$$
have $11$-good reps.

Let $\calM$ denote the list of 1252 one cusped manifolds with no hidden
symmetries, and $\calM_p$ the sublist of those with $p$-good reps for 
$p=7,11$. If $|X|$ denotes the number of elements of a set $X$, a computation
shows that

\begin{equation}
\label{eq.data}
|\calM_7 \cap \calM_{11}|=809, \quad |\calM_7\setminus\calM_{11}|=165, 
\quad |\calM_{11}\setminus\calM_{7}|=220, \quad |\calM\setminus(\calM_7
\cup  \calM_{11})|=58 \,.
\end{equation}
The manifolds in $\calM_7 \cap \calM_{11}$ with at most 10 ideal tetrahedra are
$$
v3205, v3238, v3522, K10n10, K11n27, K11n116, K12n318, K12n644.
$$
The complete data (in \texttt{SnapPy} readable format) is available from 
\cite{Ga:data}.


\section{Final comments}

In this final section we discuss further the nature of the discrete
spectrum for $1$-cusped hyperbolic 3-manifolds.  As described in
Section 2, one issue in the cusped setting is whether there is any
\emph{interesting discrete spectrum}. Theorem \ref{venkov} gives
conditions when the discrete spectrum is infinite, and we will take
this up here for $1$-cusped hyperbolic 3-manifolds.  In what follows
$M={\bf H}^3/\Gamma$ will denote a $1$-cusped orientable finite volume
hyperbolic 3-manifold with discrete spectrum $\lambda_1 \leq \lambda_2 \ldots$.

\subsection{Essentially cuspidal manifolds} 

As was mentioned previously, there is no direct analogue of the Weyl
law for cusped hyperbolic 3-manifolds, however the following
asymptotic that takes account of a contribution from the continuous
spectrum can be established using the Selberg trace formula (see
\cite{EGM} Chapter 6.5 and \cite{Sa2}). To state this, we introduce
the following notation:

For $T>0$ let $A(\Gamma,T) = |\{j : \lambda_j \leq T^2\}|$  and 
$M(\Gamma,T) = -\frac{1}{2\pi} \int_{-T}^T\frac{\phi'}{\phi}(1+it) dt$, 
then
$$
A(\Gamma,T) + M(\Gamma, T) \approx  \frac{1}{6 \pi^2} \vol(M)T^3
\quad \text{as} \quad T\rightarrow \infty \,.
$$

\noindent Therefore, getting good control on the growth of $M(\Gamma,T)$ 
implies a \emph{Weyl law}

\begin{equation}
\tag{$\dagger$}
\label{dagger}
A(\Gamma,T)  \approx \frac{1}{6 \pi^2}
\vol(M)T^3, \quad \text{as} \quad T\rightarrow \infty \,.
\end{equation}

In \cite{Sa2}, Sarnak defines $\Gamma$ or $M$ to be {\em essentially
  cuspidal} if the Weyl law $\eqref{dagger}$ holds.  Thus the issue as to 
whether $M$
is essentially cuspidal is, which of the terms $A(\Gamma,T)$ or
$M(\Gamma, T)$ dominates in the expression $\eqref{dagger}$ above.  It is
known that congruence subgroups of Bianchi groups are essentially
cuspidal (see \cite{Rez}); in this case $M(\Gamma,T) = O(T\log T)$.  
An example of a
non-congruence subgroup of a Bianchi group that is also essentially
cuspidal is given in \cite{Ef}. 

In this regard, Sarnak \cite{Sa2} has
conjectured, in a much broader context than discussed here,
that {\em if $M$ is essentially cuspidal then $M$ is
arithmetic}.  In fact, in the case of surfaces, it is conjectured
(see \cite{Sa1}) that \emph{the generic} $\Gamma$ in a given
Teichm\"uller space is not essentially cuspidal, and indeed (apart from
the case of the $1$-punctured torus) 
\emph{the generic} case should have only finitely many
discrete eigenvalues.  This is based on work of Philips and Sarnak \cite{PS}
on how eigenvalues \emph{dissolve} under deformation.  

\subsection{Knot complements} 
\label{sub.knot.complements}

Even though Theorem \ref{venkov} produces 
non-arithmetic 1-cusped hyperbolic 3-manifolds for which $A(\Gamma,T)$ is 
unbounded, the contribution from $M(\Gamma,T)$ is conjecturally enough to 
violate the Weyl law.  Now there is no analogue of the Philips and Sarnak 
result in dimension $3$, but it seems interesting to understand how the 
discrete spectrum behaves, for example for knot complements in $S^3$.

To that end, the \emph{generic knot complement} will be the minimal element
in its commensurability class, and so will likely have only finitely many 
discrete eigenvalues.  In particular, we cannot apply Theorem \ref{venkov}
to deduce an infinite discrete spectrum.

The figure-eight knot complement is the only arithmetic knot
complement, and it is also known to be a congruence manifold.  Hence,
the complement of the figure-eight knot is essentially cuspidal.
Thus, given Sarnak's conjecture, the figure-eight knot should be the
only knot complement that is essentially cuspidal. We cannot prove this
at present, but we can prove Theorem \ref{main_fig8} which we restate below 
for convenience.

\begin{theorem}
\label{main_fig8_2}
Let $M$ denote the complement of the figure-eight knot in $S^3$. Suppose
that $N$ is a finite volume hyperbolic 3-manifold which is isospectral with
$M$. Then $N$ is homeomorphic to $M$.
\end{theorem}

\begin{proof} 
Since $N$ is isospectral with $M$, $N$ cannot be closed since the
poles of the scattering function are part of the spectral data.
The result will follow once the following two claims 
are established.\\[\baselineskip]
\noindent{\bf Claims:}~{\em (1) $\Vol(N)=\Vol(M)$.}

\smallskip

\noindent (2) {\em $N$ and $M$ have the same set of lengths of closed
geodesics (without counting multiplicities).}\\[\baselineskip]
Deferring discussion of these for now, we complete the proof. From Claim
(1) and \cite{CM} the only possibility for $N$ is the so-called sister 
of the figure-eight knot. However, as can be checked by \texttt{Snap} for
example the shortest length of a closed geodesic in the sister is
approximately $0.86255\ldots$ and for the figure-eight knot complement it is 
$1.08707\ldots$. In Section \ref{shortlength}
we include a theoretical proof of the fact that the shortest
geodesic in the sister has length $0.86255\ldots$ and that the figure-eight
knot complement contains no closed geodesic of 
that length. \qed\\[\baselineskip]
Note that both (1) and (2) are standard applications of the Weyl Law and trace
formula in the setting of closed hyperbolic 3-manifolds (see for
example \cite{EGM} Chapter 5.3).  This is the approach taken here,
however, as we have already remarked upon, the cusped setting provides
additional challenges.  The proof of Claim (2) is given in
subsection \ref{dubi} and was kindly provided by Dubi Kelmer.

For Claim (1), the Weyl Law in the cusped setting takes the form
(see \cite{CoS} Chapter 7)

$$
A(\Gamma,T) + M(\Gamma, T) = 
\frac{1}{6 \pi^2} \vol(M)T^3 + O(T^2) + O(T\log~T).
$$
In the case at hand, for both $M$ and $N$ the left hand side is the same,
and so it follow that we can read off the volume (on letting 
$T\rightarrow \infty$). A different proof of equality of volume is
given in Section \ref{dubi}.
\end{proof}

Using \texttt{Snap} and \cite{GMM} we can prove the following by a similar
method.  We begin by recalling Theorem 7.4 of \cite{GMM}.

\begin{theorem}
\label{GabaiMM}
There are only ten finite volume orientable $1$-cusped hyperbolic 
$3$-manifolds with volume $\leq 2.848$. These are (in the notation
of the original \texttt{SnapPea} census):
$$
m003, m004, m006, m007, m009, m010, m011, m015, m016, m017.
$$
\end{theorem}

Note that $m003$ and $m004$ are the sister of the figure-eight knot and the
figure-eight knot respectively, $m006$ and $m007$ have the same volume
(approximately $2.56897\ldots$), $m009$ and $m010$ have the same
volume (approximately $2.66674\ldots$) and $m015$, $m016$, and $m017$
have the same volume (approximately $2.82812\ldots$). 

\begin{theorem}
\label{ten}
Let $M$ be any one the ten manifolds stated in Theorem \ref{GabaiMM}. Then
if $N$ is an orientable finite volume hyperbolic 3-manifold isospectral
with $M$ than $N$ is homeomorphic to $M$.
\end{theorem}

\begin{proof} 
As in the proof of Theorem \ref{main_fig8_2}, the manifold $N$ must have 
cusps, and by \cite[Thm.2]{Kelmer} $N$ must have $1$ cusp. As before $N$ 
also has the same volume as $M$.
Note that all 10 manifolds in the above list have fundamental group
that is $2$-generator, and so the manifolds admit an orientation-preserving
involution.
Hence Theorem \ref{venkov} applies to show that the discrete spectrum in 
all these cases is infinite.
If $N$ is
isospectral to any one of the manifolds in the list then $N$ has the
same volume.  
Theorem \ref{main_fig8_2} deals with $m004$, and also
$m003$.  Since $m011$ is the unique manifold of that volume, then this
one is also accounted for.  The only possibilities that remain to be
distinguished are the pairs $(m006,m007)$, $(m009,m010)$ and the triple
$(m015,m016,m017)$.  This can be done using \texttt{snap}
to compute the start of the length spectrum. To deal with $m006$ and $m007$, 
and $m009$ and $m010$ one can use the second shortest geodesic. To 
distinguish $m015$ from $m016$ and $m017$ 
one can use the second shortest geodesics, and $m016$ and $m017$  are 
distinguished by the shortest geodesic. 
\end{proof}

Note that $m015$ is the knot $5_2$ in the standard tables and $m016$ is
the $(-2,3,7)$-pretzel knot, and so these knots, like the figure-eight knot,
have complements that are determined by their spectral data.

\subsection{Shortest length geodesics in the sister of the figure-eight 
knot}
\label{shortlength}

Here we give a theoretical proof of the distinction in the lengths
of the shortest closed geodesic in $M$ (as above) and its sister manifold $N$.
In what follows we let $M={\bf H}^3/\Gamma_1$ and $N={\bf H}^3/\Gamma_2$ 
As is well known $\Gamma_1,\Gamma_2 < \PSL(2,\BZ[\omega])$ of index $12$,
and where $\omega^2+\omega+1=0$.

As can be easily shown (see for example \cite[Thm.4.6]{NR}),  
the shortest translation
length of a loxodromic element in $\PSL(2,\BZ[\omega])$ occurs for an element
of trace $\omega$ or its complex conjugate $\overline{\omega}$ (up to sign) 
and is approximately $0.8625546276620610\ldots$; i.e. the length of the 
shortest closed geodesic in $N$.

Fix the following elements of trace $\omega$ and $\overline{\omega}$ 
(up to sign):

$$\gamma_0 = \left(\begin{array}{cc}
 0 & 1 \\
 -1 & \omega \\
\end{array}
\right), \gamma_0' = \left(\begin{array}{cc}
 0 & -1 \\
 1 & \omega \\
\end{array}
\right)
$$
and
$$
\gamma_1 = \left(\begin{array}{cc}
 0 & 1 \\
 -1 & \overline{\omega} \\
\end{array}\right), \gamma_1' = \left(\begin{array}{cc}
 0 & -1 \\
 1 & \overline{\omega} \\
\end{array}\right).$$

As can be checked for $i=0,1$,
$\gamma_i$ and $\gamma_i'$ are not conjugate in
$\PSL(2,\BZ[\omega])$ (e.g. using reduction modulo the $\BZ[\omega]$-ideal
$<\sqrt{-3}>$).

\begin{lemma}
\label{2conjugacy}
For $i=0,1$, $\gamma_i$ and $\gamma_i'$ are representatives of all
the $\PSL(2,\BZ[\omega])$-conjugacy classes of elements
of trace $\omega$ or $\overline{\omega}$ (up to sign). 
\end{lemma}

\begin{proof} 
Suppose that $t + t^{-1} = \omega$ with 
$t = (\omega + \theta)/2$ where $\theta =
  \sqrt{-9-{\sqrt{-3}} \over 2}$ and let $k={\BQ}(\theta)$. 
It can be checked that
$k$ has discriminant $189$ and has class number one.
Using this and the formulae in Chapter III.5 of \cite{Vig2} one 
deduces that the number of conjugacy classes of elements of $\PSL(2,O_3)$ 
of trace $\omega$ is $2$.

Since an element of trace $\overline{\omega}$ simply gives a conjugate
of $k$ given by ${\bf Q}(\overline{t})$, the same argument applies to
also give two conjugacy classes in this case.
\end{proof}

The claim about the lengths will follow once we establish that none of the 
$\PSL(2,\BZ[\omega])$-conjugacy classes of  
$\gamma_i$ and $\gamma_i'$ for $i=0,1$, meet $\Gamma_1$ and at least one
meets $\Gamma_2$. This can be done efficiently using \texttt{magma} as we
now describe. We begin with a preliminary observation.

Suppose that $M={\bf H}^3/\Gamma\rightarrow Q={\bf H}^3/\Gamma_0$ is
a finite sheeted covering of finite-volume orientable hyperbolic 3-orbifolds.
Denoting the covering degree by $d$, the action on cosets of $\Gamma$ in 
$\Gamma_0$ determines a permutation representation 
$\rho:\Gamma_0\rightarrow S_d$
with kernel $K$. Suppose that $[g_1],\ldots ,[g_r]$ denote the conjugacy
classes of loxodromic elements in $\Gamma_0$ of minimal translation
length $\ell$. Then $M$ contains an element of length $\ell$ if and only
if $\Gamma \cap [g_i]\neq \emptyset$ for some $i\in\{1,\ldots ,r\}$, and this
happens if and only $\rho(\Gamma)\cap [\rho(g_i)]\neq \emptyset$
for some $i\in\{1,\ldots ,r\}$.

We apply this in the case that $N$ is the Bianchi orbifold 
$Q={\bf H}^3/\PSL(2,\BZ[\omega])$ and $M$ is either the figure-eight knot 
complement or
its sister.  In the former case, the permutation representation has kernel
the congruence subgroup $\Gamma(4)<\PSL(2,\BZ[\omega])$ (of index $1920$) 
and in the latter
case the permutation representation has kernel
the congruence subgroup $\Gamma(2)$ (of index $60$).  To implement
the magma routines we use the the presentation
of $\PSL(2,\BZ[\omega])$ from \cite{GS}, and express the subgroups $\Gamma_1$
and $\Gamma_2$ in terms of these generators. Setting

\smallskip

\centerline{$a=\left(\begin{array}{cc}
 1 & 1 \\
 0 & 1 \\
\end{array}\right),$
$b=\left(\begin{array}{cc}
 0 & -1 \\
 1 & 0 \\
\end{array}\right),$ and
$c=\left(\begin{array}{cc}
 1 & \omega \\
 0 & 1 \\
\end{array}\right)$,}
we have
\begin{align*}
\PSL(2,\BZ[\omega]) &= <a,b,c|
b^2 = (ab)^3 = (acbC^{2}b)^2 = (acbCb)^3 =
A^{2}CbcbCbCbcb = [a,c] = 1> \\
\Gamma_1 &= <a,bcb> \\
\Gamma_2 &= <a^2, bcabaCbCb>.
\end{align*}

The elements $\gamma_i$ and $\gamma_i'$ for $i=0,1$ are described in terms 
of these generators as:

$$
\gamma_0=bC, \gamma_0'=bc,\gamma_1=bac,\gamma_1'=bAC.
$$

Below we include the magma routine that executes the above computation showing
no conjugates lie in $\Gamma_1$ but at least one does in $\Gamma_2$.

\begin{verbatim}
g<a,b,c>:=Group<a,b,c| b^2, (a*b)^3, (a*c*b*c^-2*b)^2,
(a*c*b*c^-1*b)^3, a^-2*c^-1*b*c*b*c^-1*b*c^-1*b*c*b, (a,c)>;
h1:= sub<g|a,b*c*b>;
h2:= sub<g|a^2, b*c*a*b*a*c^-1*b*c^-1*b>;
print AbelianQuotientInvariants(h1);
\\[0]
print AbelianQuotientInvariants(h2);
\\[ 5, 0 ]
x0:=g!b*c^-1;
x1:=g!b*c;
y0:=g!b*a*c;
y1:=g!b*a^-1*c^-1;
f1,i1,k1:=CosetAction(g,h1);
print Order(i1);
\\1920
f2,i2,k2:=CosetAction(g,h2);
print Order(i2);
\\60
l:=Class(i1,f1(x0)) meet Set(f1(h1));
print #l;                            
\\0
l:=Class(i1,f1(x1)) meet Set(f1(h1));
print #l;                            
\\0
l:=Class(i1,f1(y0)) meet Set(f1(h1));
print #l;                            
\\0
l:=Class(i1,f1(y1)) meet Set(f1(h1));
print #l;                            
\\0
k:=Class(i2,f2(x0)) meet Set(f2(h2)); 
print #k;                            
\\2
\end{verbatim}

\subsection{Determining the length set}
\label{dubi}

\begin{proposition}
\label{dubiproof}
Let $M_1$ and $M_2$ be finite volume orientable 1-cusped hyperbolic 3-manifolds
that are isospectral.  Then they have the same set of lengths of closed
geodesics (without counting multiplicities).
\end{proposition}

Before commencing with the proof we recall the version of the trace
formula given in \cite{EGM} Chapter 6 Theorem 5.1. This needs some notation. 
Let $M={\bf H}^3/\Gamma$ be $1$-cusped finite volume orientable hyperbolic
3-manifold. Given a loxdromic element $\gamma\in \Gamma$ of complex
length $\ell_\gamma+i\theta_\gamma$, we denote by $\gamma_0$ the
unique primitive element such that $\gamma=\gamma_0^j$.
For convenience, we denote the discrete spectrum by 
$\lambda_k=1+r_k^2$, and by $\phi(s)$ is (as before) the scattering function.
The trace formula in this case then states (\cite{EGM} Chapter 6 Theorem 5.1):

\begin{theorem}
For any even compactly supported test function $g\in C^\infty_c(\BR)$ let 
$h(r)$ denote it's Fourier transform. Then
\begin{align*}
\sum_{k}h(r_k)-\frac{1}{4\pi}\int_{\BR}h(r)\frac{\phi'}{\phi}(ir)dr
& =\frac{\vol(M)}{4\pi^2}\int_{\BR}h(t)r^2dr\\
&+4\pi \sum_{\gamma\in \Gamma_{\mathrm{lox}}}\frac{\ell_{\gamma_0} 
g(\ell_{\gamma})}{\sqrt{2\sinh(\frac{\ell_\gamma+i\theta_\gamma}{2})}}
+a_\Gamma g(0)+b_\Gamma h(0)\\
&-\frac{1}{2\pi}\int_{\BR}h(r)\frac{\Gamma'}{\Gamma}(1+ir)dr 
\end{align*}
where the constants $a_\Gamma$ and $b_\Gamma$ are explicit constants depending 
only on $\Gamma$ and the summation on the right-hand side is over
$\Gamma_{\mathrm{lox}}$ which represents conjugacy classes of loxodromic 
elements in $\Gamma$.
\end{theorem}

\begin{remark}
Our notation is slightly different from \cite{EGM} and there are less terms 
due to our assumptions of only one cusp and no torsion. Also, the 
$\Gamma(s)$ in the last integral denotes the $\Gamma$-function and is not 
related to the Kleinian group $\Gamma$.
\end{remark}

We now prove Proposition \ref{dubiproof}.

\begin{proof}
For our purposes it will be helpful to rewrite the sum over the
loxodromic classes and collect all the classes with the same
$\ell_\gamma$ together. That is,
\begin{align*} 
\sum_{\gamma\in \Gamma_{\mathrm{lox}}}
\frac{\ell_{\gamma_0} g(\ell_{\gamma})}{\sqrt{2\sinh(
\frac{\ell_\gamma+i\theta_\gamma}{2})}}
&=
\sum_{\ell}\left(\mathop{\sum_{\gamma\in \Gamma_{\mathrm{lox}}}}_{\ell_\gamma=\ell}
\frac{\ell_{\gamma_0} }{\sqrt{2\sinh(\frac{\ell_\gamma+i\theta_\gamma}{2})}}
\right) g(\ell)\\
&=\sum_{\ell} m_\Gamma(\ell)g(\ell)\
\end{align*}
where we defined the twisted multiplicities $m_\Gamma(\ell)$ by
$$
m_\Gamma(\ell)=\mathop{\sum_{\gamma\in \Gamma_{\mathrm{lox}}}}_{\ell_\gamma=\ell}
\frac{\ell_{\gamma_0} }{\sqrt{2\sinh(\frac{\ell_\gamma+i\theta_\gamma}{2})}},
$$
and the sum on the right is over the set of lengths of closed geodesics in 
$M$ (in fact we can take the sum over all $\ell>0$ since $m_\Gamma(\ell)=0$ 
if $\ell$ is not a length of a closed geodesic).

We can thus rewrite the trace formula as 
\begin{align*}
\sum_{k}h(r_k)-\frac{1}{4\pi}\int_{\BR}h(r)\frac{\phi'}{\phi}(ir)dr
&=\frac{\vol(M)}{4\pi^2}\int_{\BR}h(t)r^2dr
+4\pi \sum_{\ell}m_\Gamma(\ell)g(\ell)\\
&+a_\Gamma g(0)+b_\Gamma h(0)-\frac{1}{2\pi}
\int_{\BR} h(r)\frac{\Gamma'}{\Gamma}(1+ir)dr
\end{align*}

Noting that $m_\Gamma(\ell)\neq 0$ if and only if $\ell$ is a length of a 
closed geodesic in $M$, the result will follow from the next proposition.

\begin{proposition}
\label{dubi2}
With $M_1$ and $M_2$ as in Proposition \ref{dubi}, then
$\Vol(M_1)=\Vol(M_2)$, 
$a_{\Gamma_1}=a_{\Gamma_2}$, $b_{\Gamma_1}=b_{\Gamma_2}$, and 
$m_{\Gamma_1}(\ell)=m_{\Gamma_2}(\ell)$ for any $\ell>0$.  
\end{proposition}

\begin{proof}
Let $\Delta V=\Vol(M_1)-\Vol(M_2)$, $\Delta a=a_{\Gamma_1}-a_{\Gamma_2}$, 
$\Delta b=b_{\Gamma_1}-b_{\Gamma_2}$, and 
$\Delta m(\ell)=m_{\Gamma_1}(\ell)-m_{\Gamma_2}(\ell)$. Taking the difference 
between the two trace formulas, the left hand side cancels and we get that 
for any even test function $g\in C^\infty_c(\BR)$ 
\begin{align*}
\frac{\Delta V}{4\pi^2}\int_{\BR} h(t)r^2dr
+4\pi \sum_{\ell}\Delta m(\ell)g(\ell)+\Delta a g(0)+\Delta b h(0)&=0
\end{align*}
We can first take a test function $g$ to be supported away from all
the lengths in the length spectrum of both manifolds and from 0 (e.g.,
make it supported in the interval between zero and the shortest
length), and satisfy that $h(0)=\int g(x)dx=0$ but 
$\int_{\BR}h(t)r^2dr\neq 0$. Using such a test function we can deduce that
$\Delta V=0$ (which was already deduced from Weyl's law). The
difference of the trace formula hence simplifies to
\begin{align*}
4\pi \sum_{\ell}\Delta m(\ell)g(\ell)+\Delta a g(0)+\Delta b h(0)&=0
\end{align*}
Next, taking $g$ supported away from all lengths and 0 but this time with 
$h(0)=1$ we conclude that $\Delta b=0$, and then taking $g$ supported on 
a small neighborhood of $0$ (smaller than the length of the shortest 
geodesic) we conclude that $\Delta a=0$ as well. From this we get that for 
any test function
\begin{align*}
\sum_{\ell}\Delta m(\ell)g(\ell)&=0
\end{align*}
Finally, for each $\ell>0$ we can take $g$ to be supported in a small 
enough neighborhood of $\ell$, such that no other length in the length 
spectrum are in the support (except $\ell$ itself if it happens to be in 
the length spectrum of one of the manifolds).
This implies that $\Delta m(\ell)=0$ as well for any $\ell>0$, thus 
concluding the proof.
\end{proof}

The proof of Proposition \ref{dubiproof} is now complete.
\end{proof}

\subsection*{Acknowledgments}
{\em The authors wish to thank Nathan Dunfield, Craig Hodgson, Neil Hoffman,
and Dubi Kelmer for enlightening communications. 
The second author would also like to thank Amir Mohammadi
for several very useful conversations, Werner M\"uller for several very
helpful conversations on the notion of isospectrality for cusped manifolds,
and Peter Sarnak for many helpful discussions and correspondence on the 
nature of the discrete spectrum and arithmeticity. We would particularly
like to thank Dubi Kelmer who supplied the details of the proof
of Proposition \ref{dubiproof}. 
The second author also 
wishes to thank Max Planck Instit\"ut, Bonn where some of this work was 
carried out.
}


\bibliographystyle{plain}
\bibliography{biblio}

\def\cprime{$'$}
\begin{thebibliography}{10}

\bibitem{Ber}
Pierre B{\'e}rard.
\newblock On the construction of isospectral {R}iemannian manifolds.
\newblock Notes 1993.

\bibitem{Br}
Robert Brooks.
\newblock Constructing isospectral manifolds.
\newblock {\em Amer. Math. Monthly}, 95(9):823--839, 1988.

\bibitem{Br1}
Robert Brooks.
\newblock The {S}unada method.
\newblock In {\em Tel {A}viv {T}opology {C}onference: {R}othenberg
  {F}estschrift (1998)}, volume 231 of {\em Contemp. Math.}, pages 25--35.
  Amer. Math. Soc., Providence, RI, 1999.

\bibitem{BroD}
Robert Brooks and Orit Davidovich.
\newblock Isoscattering on surfaces.
\newblock {\em J. Geom. Anal.}, 13(1):39--53, 2003.

\bibitem{CM}
Chun Cao and G.~Robert Meyerhoff.
\newblock The orientable cusped hyperbolic {$3$}-manifolds of minimum volume.
\newblock {\em Invent. Math.}, 146(3):451--478, 2001.

\bibitem{CoS}
Paul Cohn and Peter Sarnak.
\newblock Notes on the {S}elberg trace formula, {S}tanford {U}niversity, 1980.
\newblock \url{http://publications.ias.edu/sarnak}.

\bibitem{snappy}
Marc Culler, Nathan~M. Dunfield, and Jeffery~R. Weeks.
\newblock {S}nap{P}y, a computer program for studying the geometry and topology
  of 3-manifolds.
\newblock \url{http://snappy.computop.org}.

\bibitem{Ef}
Isaac Efrat.
\newblock On the automorphic forms of a noncongruence subgroup.
\newblock {\em Michigan Math. J.}, 34(2):217--226, 1987.

\bibitem{ES}
Isaac Efrat and Peter Sarnak.
\newblock The determinant of the {E}isenstein matrix and {H}ilbert class
  fields.
\newblock {\em Trans. Amer. Math. Soc.}, 290(2):815--824, 1985.

\bibitem{EGM}
J{\"u}rgen Elstrodt, Fritz Grunewald, and Jens Mennicke.
\newblock {\em Groups acting on hyperbolic space}.
\newblock Springer Monographs in Mathematics. Springer-Verlag, Berlin, 1998.
\newblock Harmonic analysis and number theory.

\bibitem{census-tet}
Evgeny Fominykh, Stavros Garoufalidis, Matthias Goerner, Vladimir Tarkaev, and
  Andrei Vesnin.
\newblock A {C}ensus of {T}etrahedral {H}yperbolic {M}anifolds.
\newblock {\em Exp. Math.}, 25(4):466--481, 2016.

\bibitem{GMM}
David Gabai, Robert Meyerhoff, and Peter Milley.
\newblock Mom technology and hyperbolic 3-manifolds.
\newblock In {\em In the tradition of {A}hlfors-{B}ers. {V}}, volume 510 of
  {\em Contemp. Math.}, pages 84--107. Amer. Math. Soc., Providence, RI, 2010.

\bibitem{Ga:data}
Stavros Garoufalidis.
\newblock
  \url{http://www.math.gatech.edu/~stavros/publications/isospectral.data}.

\bibitem{snap}
Oliver Goodman.
\newblock {S}nap, the computer program.
\newblock \url{http://www.ms.unimelb.edu.au/~snap}.

\bibitem{GHH}
Oliver Goodman, Damian Heard, and Craig Hodgson.
\newblock Commensurators of cusped hyperbolic manifolds.
\newblock {\em Experiment. Math.}, 17(3):283--306, 2008.

\bibitem{GPS}
Carolyn Gordon, Peter Perry, and Dorothee Schueth.
\newblock Isospectral and isoscattering manifolds: a survey of techniques and
  examples.
\newblock In {\em Geometry, spectral theory, groups, and dynamics}, volume 387
  of {\em Contemp. Math.}, pages 157--179. Amer. Math. Soc., Providence, RI,
  2005.

\bibitem{Gor}
Carolyn~S. Gordon.
\newblock Survey of isospectral manifolds.
\newblock In {\em Handbook of differential geometry, {V}ol. {I}}, pages
  747--778. North-Holland, Amsterdam, 2000.

\bibitem{GS}
Fritz Grunewald and Joachim Schwermer.
\newblock Subgroups of {B}ianchi groups and arithmetic quotients of hyperbolic
  {$3$}-space.
\newblock {\em Trans. Amer. Math. Soc.}, 335(1):47--78, 1993.

\bibitem{Gu}
Robert~M. Guralnick.
\newblock Subgroups inducing the same permutation representation.
\newblock {\em J. Algebra}, 81(2):312--319, 1983.

\bibitem{Sa1}
Dennis~A. Hejhal, Peter Sarnak, and Audrey~Anne Terras, editors.
\newblock {\em The {S}elberg trace formula and related topics}, volume~53 of
  {\em Contemporary Mathematics}. American Mathematical Society, Providence,
  RI, 1986.

\bibitem{Kac}
Mark Kac.
\newblock Can one hear the shape of a drum?
\newblock {\em Amer. Math. Monthly}, 73(4, part II):1--23, 1966.

\bibitem{Kelmer}
Dubi Kelmer.
\newblock On distribution of poles of {E}isenstein series and the length
  spectrum of hyperbolic manifolds.
\newblock {\em Int. Math. Res. Not. IMRN}, (23):12319--12344, 2015.

\bibitem{Mu}
Werner M{\"u}ller.
\newblock Spectral geometry and scattering theory for certain complete surfaces
  of finite volume.
\newblock {\em Invent. Math.}, 109(2):265--305, 1992.

\bibitem{NR}
Walter~D. Neumann and Alan~W. Reid.
\newblock Arithmetic of hyperbolic manifolds.
\newblock In {\em Topology '90 ({C}olumbus, {OH}, 1990)}, volume~1 of {\em Ohio
  State Univ. Math. Res. Inst. Publ.}, pages 273--310. de Gruyter, Berlin,
  1992.

\bibitem{PS}
Ralph~S. Phillips and Peter Sarnak.
\newblock On cusp forms for co-finite subgroups of {${\rm PSL}(2,{\bf R})$}.
\newblock {\em Invent. Math.}, 80(2):339--364, 1985.

\bibitem{PR}
Dipendra Prasad and Conjeeveram~S. Rajan.
\newblock On an {A}rchimedean analogue of {T}ate's conjecture.
\newblock {\em J. Number Theory}, 99(1):180--184, 2003.

\bibitem{Reid}
Alan~W. Reid.
\newblock Isospectrality and commensurability of arithmetic hyperbolic {$2$}-
  and {$3$}-manifolds.
\newblock {\em Duke Math. J.}, 65(2):215--228, 1992.

\bibitem{Rez}
Andrei Reznikov.
\newblock Eisenstein matrix and existence of cusp forms in rank one symmetric
  spaces.
\newblock {\em Geom. Funct. Anal.}, 3(1):79--105, 1993.

\bibitem{Ri}
Robert Riley.
\newblock Parabolic representations of knot groups. {I}.
\newblock {\em Proc. London Math. Soc. (3)}, 24:217--242, 1972.

\bibitem{Salvai}
Marcos Salvai.
\newblock On the {L}aplace and complex length spectra of locally symmetric
  spaces of negative curvature.
\newblock {\em Math. Nachr.}, 239/240:198--203, 2002.

\bibitem{Sa2}
Peter Sarnak.
\newblock On cusp forms. {II}.
\newblock In {\em Festschrift in honor of {I}. {I}. {P}iatetski-{S}hapiro on
  the occasion of his sixtieth birthday, {P}art {II} ({R}amat {A}viv, 1989)},
  volume~3 of {\em Israel Math. Conf. Proc.}, pages 237--250. Weizmann,
  Jerusalem, 1990.

\bibitem{Su}
Toshikazu Sunada.
\newblock Riemannian coverings and isospectral manifolds.
\newblock {\em Ann. of Math. (2)}, 121(1):169--186, 1985.

\bibitem{pari}
{The PARI~Group}, Bordeaux.
\newblock {\em {PARI/GP, version {\tt 2.5.3}}}, 2012.
\newblock available from \url{http://pari.math.u-bordeaux.fr/}.

\bibitem{Ven1}
Alexei~B. Venkov.
\newblock The space of cusp functions for a {F}uchsian group of the first kind
  with a nontrivial commensurator.
\newblock {\em Dokl. Akad. Nauk SSSR}, 239(3):511--514, 1978.

\bibitem{Vig2}
Marie-France Vign{\'e}ras.
\newblock {\em Arithm\'etique des alg\`ebres de quaternions}, volume 800 of
  {\em Lecture Notes in Mathematics}.
\newblock Springer, Berlin, 1980.

\bibitem{Vig}
Marie-France Vign{\'e}ras.
\newblock Vari\'et\'es riemanniennes isospectrales et non isom\'etriques.
\newblock {\em Ann. of Math. (2)}, 112(1):21--32, 1980.

\end{thebibliography}
\end{document}